\documentclass[11pt]{article}

\usepackage{amsmath,amsfonts,amsthm,amssymb,verbatim,graphicx,subfigure,color,fullpage,multirow}
\usepackage{pgf,tikz,pgfplots} 

\usepackage{hyperref}
\hypersetup{
    colorlinks=true,
    linkcolor=blue,
    filecolor=magenta,      
    urlcolor=blue,
    pdftitle={Overleaf Example},
    pdfpagemode=FullScreen,
    }

\usepackage{tikz}
\usetikzlibrary{shapes.geometric, arrows,decorations.pathreplacing,calc,arrows}
\usetikzlibrary{mindmap}

\usepackage{colortbl}
\usepackage{xcolor}
\definecolor{gris}{rgb}{.5,.5,.5}

\usepackage{pgfplots}
\usepackage{color}

\definecolor{0}{RGB}{117,170,219}
\definecolor{1}{rgb}{1,0.4,0.4}

\newtheorem{lemma}{Lemma}
\newtheorem{theorem}{Theorem}
\newtheorem*{theorem*}{Theorem}

\newtheorem{definition}{Definition}
\newtheorem{claim}{Claim}

\title{$s$-Stable Kneser Graph are Hamiltonian
}
\author{Agustina V. Ledezma\footnote{agustina.ledezma.tag@gmail.com}~
and Adri\'{a}n G. Pastine\footnote{adrian.pastine.tag@gmail.com}\\
Instituto de Matem\'atica Aplicada San Luis (UNSL-CONICET),\\
\&
Departamento de Matem\'atica,
Universidad Nacional de San Luis, \\
Ej\'ercito de los Andes 950 (D5700HHW),\\
San Luis, Argentina\\}

\begin{document}
\maketitle
\sloppy

\begin{abstract}
The Kneser Graph $K(n,k)$ has as vertices all $k$-subsets of $\{1,\ldots,n\}$ and edges connecting two vertices if they are disjoint. The $s$-stable Kneser  Graph 
$K_{s-stab}(n, k)$ is obtained from the Kneser graph by deleting vertices with elements at cyclic distance less than $s$. In this article we show that connected $s$-Stable Kneser graph are Hamiltonian.

\noindent {\bf Keywords}: Kneser graphs, Hamiltonian graphs, Stable Kneser Graphs, Schrijver Graphs.
\end{abstract}

\section{Introduction}
Let $[n] = \lbrace 1,...,n \rbrace $. For each $ n \geq 2k$, $n,k \in \lbrace 1,2,3, \dots \rbrace$, the \textbf{Kneser Graph} $K(n, k)$ has as vertices the set $  {[n] \choose k} = \lbrace S \subseteq [n] : \ |S| = k \rbrace$, the k-subsets of $[n]$, where two vertices are adjacent if they are disjoint.
Kneser graphs have been widely studied throughout the literature. Two of the most important problems in the history of Kneser graphs are its chromatic number (proven by Lov\'{a}sz \cite{Lov} to be $n-k+2$) and their Hamiltonicity. It was long conjectured that, with the exception of $K(5,2)$, all connected Kneser graphs are Hamiltonian. Many articles studied this problem (see \cite{Chen}, \cite{Mut}, \cite{MNW} and \cite{SS}) until recently (2023), when Merino, M{\"u}tze, and Namrata showed in \cite{MMN}  that the conjecture is true.

In \cite{Sch}, Schrijver introduced a family of subgraphs of Kneser graphs which are vertex critical in terms of their chromatic number. These graphs received the name of Schrijver graphs, are denoted by $SG(n,k)$, and are obtained from the Kneser graph $K(n,k)$ by deleting vertices $S\subset {[n] \choose k}$ containing consecutive elements modulo $n$. Schrijver showed that the chromatic number of $SG(n,k)$ is $n-k+2$, and  that deleting any vertex from $SG(n,k)$ reduces its chromatic number.
This family of graphs was generalized in \cite{defestables}, where Alon, Drewnoski and Luczak introduced the concept of $s$-stable Kneser graphs. The \textbf{$s$-Stable Kneser Graph}, $K_{s-stab}(n, k)$, is the graph that has as vertex set $S\subset {[n] \choose k}$ such that $s \leq \mid i-j\mid \leq n-s$ for every pair $i,j \in$ S, and edges between disjoint vertices. Notice that $K_{s-stab}(n, k)$ is an induced subgraph of $K(n,k)$, and that $K_{2-stab}(n, k)$ and $SG(n,k)$ are the same graph. Since the family of $s$-stable Kneser graphs was introduced, both it in general, and more specifically the family of Schrijver graphs, started receiving much attention, partly due to their applications to topology (see \cite{aplicaciontopo}). Thus, several properties of this family have been studied, such as its automorphism group (see \cite{Braun} and \cite{Tor}), its chromatic number (see \cite{Meu} and \cite{TorPav}), its diameter (see \cite{paper}), hom-idempotence (see \cite{TorPav}), independence complexes (see \cite{Indepcomplex}), and neighborhood complexes (see \cite{Neighcomplexes}).

The problem of Hamiltonicity of $s$-Stable Kneser graphs began by the authors of this manuscript in 2019, and earlier results were presented at  annual meeting of Uni\'on Matem\'atica Argentina (Argentinian Mathematical Union) in 2019 and in 2021. The main result of this article is the following.
\begin{theorem}\label{theo: main}
    Let $k\geq 1$ and $s\geq 3$.  The graph $K_{s-stab}(n, k)$ is Hamiltonian if and only if $n\geq sk$. The graph $K_{2-stab}(n, k)$ is Hamiltonian if and only if $n\geq 2k+1$.
\end{theorem}
For a conference talk (in Spanish) presenting these results, we direct the reader to
\url{https://www.youtube.com/watch?v=9ZJ9-ruKmG0}. Similar results were obtained independently by M{\"u}tze and Namrata in \cite{mutze2024hamiltonicity}.

The rest of the article is organized as follows. In Section \ref{Section:Prelim} we introduced some notation and give a little insight on the behaviour of the vertices of $K_{s-stab}(n, k)$. Then, in Section \ref{Section: class graph} we introduced some auxiliary graphs given by the orbits of the vertices of $K_{s-stab}(n, k)$ under rotation of its elements, and use them to prove our main result.

\section{Preliminaries}\label{Section:Prelim}

\begin{definition}\label{Stable Kneser Graphs}
The \textbf{$s$-Stable Kneser Graph}, $K_{s-stab}(n, k)$, has vertex set $S\subset {[n] \choose k}$ such that $s \leq \mid i-j\mid \leq n-s$ for every pair $i,j \in$ S, and edges between disjoint vertices.
\end{definition}

Let $K_{s-stab}(n, k)$,  and let $V$ be a vertex of the graph, for a better understanding of vertices and edges, we represent $V$ as a grid of $n$ squares, ordered from the first position to the position $n$ (which is connected to the first one), in which $k$ of the positions are occupied by ``$x$'' as indicated in $V$, and the remain positions are white spaces.
Then, for instance, the vertex $\lbrace 1,3,5 \rbrace$ of $K_{2-stab}(9, 3)$ is a grid in which positions 1, 3 and 5 are ``$x$''. See Figure \ref{v13}. Notice, in particular, that when $n=sk$, $K_{s-stab}(n, k)$ is isomorphic to $K_s$.

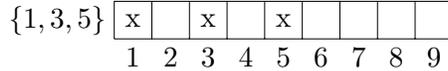
\begin{figure}
\centering 
\begin{tikzpicture}[x=.5cm,y=.5cm]
\draw[step=.5cm,very thin] (0,0) grid (9,1);
\node[draw=none,shape=rectangle] at (-1.5,.5) {$\{1,3,5\}$}; 
\node[draw=none,shape=rectangle] at (0.5,.5) {x};   
\node[draw=none,shape=rectangle] at (2.5,.5) {x};      
\node[draw=none,shape=rectangle] at (4.5,.5) {x};
\node[draw=none,shape=rectangle] at (0.5,-.5) {1};   
\node[draw=none,shape=rectangle] at (1.5,-.5) {2};   
\node[draw=none,shape=rectangle] at (2.5,-.5) {3};
\node[draw=none,shape=rectangle] at (3.5,-.5) {4};
\node[draw=none,shape=rectangle] at (4.5,-.5) {5};
\node[draw=none,shape=rectangle] at (5.5,-.5) {6};   
\node[draw=none,shape=rectangle] at (6.5,-.5) {7};   
\node[draw=none,shape=rectangle] at (7.5,-.5) {8};
\node[draw=none,shape=rectangle] at (8.5,-.5) {9};
\end{tikzpicture}
\caption{Grid of the vertex $\lbrace 1,3,5 \rbrace$ in $K_{2-stab}(9, 3)$.}
\label{v13}
\end{figure}

It is also interesting to think a vertex as a list of the spaces between the ``$x$''.
We define the class of a vertex $V \in K_{s-stab}(n, k)$ as the orbit of $V$ under the cyclic rotation of the elements of $[n]$. As this rotation defines an automorphism of $K_{s-stab}(n, k)$, it will be helpful to study when a vertex in a class is adjacent to a vertex in a different class. For this purpose we represent the class of $V$ with the list of spaces between the elements of $V$, although there is some ambiguity in this as other vertices in the same class may present a rotation of the list of spaces. As an example, 
in Figure \ref{clasesdif} we present all the elements of the classes $(1,2,3)$ and $(1,3,2)$ of vertices in $K_{2-stab}(9, 3)$. As we can notice those classes are different. Nevertheless, $(1,2,3)$, $(2,3,1)$ and $(3,1,2)$ are the same class.

\begin{figure}
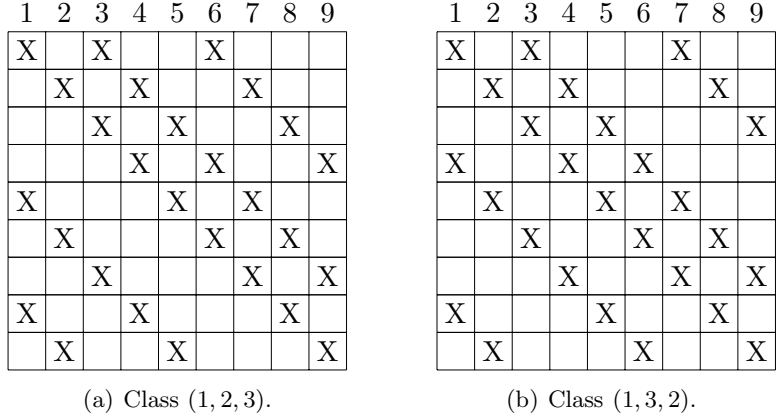

    \centering
\subfigure [t] [Class $(1,2,3)$.] 
{
\tikz[x=.5cm,y=.5cm]{
{\foreach \i in {0,1,2,3,4,5,6,7,8}
{
\pgfmathtruncatemacro{\r}{(mod(\i+2,9))};
\pgfmathtruncatemacro{\s}{(mod(\i+5,9))};
\draw[step=.5cm,very thin] (0,-1*\i) grid (9,1-1*\i);
\node[draw=none,shape=rectangle] at (\i+.5,-1*\i+0.5) {X};         
\node[draw=none,shape=rectangle] at (\r+.5,-1*\i+0.5) {X};
\node[draw=none,shape=rectangle] at (\s+.5,-1*\i+0.5) {X};
}
{\foreach \p in {0,1,2,3,4,5,6,7,8}{
\pgfmathtruncatemacro{\pp}{(mod(\p+1,15))};
\node[draw=none,shape=rectangle] at (\p+.5,1.5) {$\pp$};}}
}}
} \qquad
\subfigure [t] [Class $(1,3,2)$.] 
{
\tikz[x=.5cm,y=.5cm]{
{\foreach \i in {0,1,2,3,4,5,6,7,8}
{
\pgfmathtruncatemacro{\r}{(mod(\i+2,9))};
\pgfmathtruncatemacro{\s}{(mod(\i+6,9))};
\draw[step=.5cm,very thin] (0,-1*\i) grid (9,1-1*\i);
\node[draw=none,shape=rectangle] at (\i+.5,-1*\i+0.5) {X};         
\node[draw=none,shape=rectangle] at (\r+.5,-1*\i+0.5) {X};
\node[draw=none,shape=rectangle] at (\s+.5,-1*\i+0.5) {X};
}
{\foreach \p in {0,1,2,3,4,5,6,7,8}{
\pgfmathtruncatemacro{\pp}{(mod(\p+1,15))};
\node[draw=none,shape=rectangle] at (\p+.5,1.5) {$\pp$};}}
}}}
\caption{Grid of vertices in $K_{2-stab}(9, 3)$.}
    \label{clasesdif}

\end{figure}


Let $V= \{ a_1,a_2, \dots , a_k \}$ be a vertex in $K_{s-stab}(n, k)$, and $V+1=\{ a_1+1,a_2+1, \dots , a_k+1 \}$ the rotation of $V$. Since between every element in $V$ there are at least $s-1$ spaces, when we rotate every position there is no intersection in the elements of both vertices. As we can visualize the elements of a vertex as ``$x$'' in a grid, the rotation of a vertex is in fact rotate one place to the right (cyclically) all the ``$x$''s.
Let $\mathcal{A}$ be a class of vertices in $K_{s-stab}(n, k)$, we define its order $|\mathcal{A}|$ to be the number of vertices in the class.
Then, given a vertex $V\in \mathcal{A}$, we can form a cycle spanning all vertices in $\mathcal{A}$ by successive rotations.
\begin{claim}
For each class $\mathcal{A}$ and each vertex $V\in \mathcal{A}$, 
\[
V+1,V+2,\ldots,V+|\mathcal{A}|, V
\]
is a cycle spanning the vertices of $\mathcal{A}$.
\label{rotarvecinos da un ciclo}
\end{claim}


We denote by  $(a_1,a_2, \dots , a_i)^j$ the concatenation of $j$ copies of $a_1,a_2, \dots , a_i$. I.e., $(a_1,a_2, \dots , a_i)^j=(a_1,a_2, \dots , a_i, a_1,a_2, \dots , a_i, \dots, a_1,a_2, \dots , a_i)$.  Using this notation, we can cleanly represent the order of a class, as follows.

\begin{claim}
$|\mathcal{A}|=\dfrac{n}{d}$ if and only if $d$ is the smallest number such that $\mathcal{A} = (a_1,a_2, \dots , a_{\frac{k}{d}})^d$.
\label{cardclase}
\end{claim}


Notice that in order to have a class of order $n/d$, $d$ must divide both $n$ and $k$. Thus, it must divide the number of blank places, $n-k$. As an example, lets look at the grid of the graph $K_{3-stab}(36, 6)$ there are $36-6=30$ blank places to divide into $6$ blocks, such that each block has at least $2$ blank places. 
We know that $d$ is a common divisor between 36 and 6, if and only if $d$ is a common divisor between 30 and 6. As $1,2,3$ and $6$ are the common divisors between $36$ and $6$, there are classes of order $36,18,12 $ and $6$. In Table \ref{algunas} we can observe some of those classes.

\begin{table}[h]
    \centering
    \begin{tabular}{|l|l|}
    \hline
    Class     & Order\\
    \hline
    (2,2,2,2,11,11)  & 36\\
    $(2,2,11)^2=(2,2,11,2,2,11)$  & 36/2=18\\
    $(2,8)^3=(2,8,2,8,2,8)$  & 36/3=12\\
    $(6)^6=(6,6,6,6,6,6)$  & 36/6=6\\
    \hline
    \end{tabular}
    \caption{Some classes of the graph $K_{3-stab}(36, 6)$.}
    \label{algunas}
\end{table}

\begin{lemma}\label{vecinos entre clases}
    Let $A$ and $B$ be vertices in different classes of $K_{s-stab}(n, k)$.
    Then, $A$ and $B$ are adjacent if and only if $A+1$ and $B+1$ are adjacent.
\end{lemma}
\begin{proof}
    Let $A=\{a_1,a_2, \dots,a_k\}$ and $B=\{b_1,b_2, \dots,b_k\}$.
    If $A$ and $B$ are adjacent then $A \cap B = \emptyset$.

    Suppose that $c\in (A+1 \cap B+1) $. Then, $c=a_i+1$ and $c=b_j+1$, which implies that $a_i=b_j$. Then $A+1$ and $B+1$ are adjacent.
\end{proof}


\section{Class Graph and proof of the main theorem}\label{Section: class graph}
\begin{definition}
Let $K_{s-stab}(n, k)$ with $n\geq sk+1$, we assign every vertex to its class as we have shown before. The \textbf{Class Graph} of the $s$-Stable Kneser Graph $K_{s-stab}(n, k)$, denoted by $CK_{s-stab}(n, k)$, has as vertices the classes of the vertices of $K_{s-stab}(n, k)$, and edges as in $K_{s-stab}(n, k)$. If two vertices are neighbors in $K_{s-stab}(n, k)$, then their classes are neighbors in $CK_{s-stab}(n, k)$.

\end{definition}
\begin{claim}\label{C:si spanning de grado bajo, entonces hamiltoniano}
Let $T$ be a spanning tree of the Class Graph $CK_{s-stab}(n, k)$ with $n\geq sk+1$, $s \geq 2$. If  $\deg_T(\mathcal{A})\leq |\mathcal{A}|$ for every $\mathcal{A}\in V(CK_{s-stab}(n, k))$, then the $s$-Stable Kneser Graph $K_{s-stab}(n, k)$ is Hamiltonian.
\end{claim}
\begin{proof}
    Root $T$ at a vertex $\mathcal{R}$, and partition the vertices in levels according to their distance to $\mathcal{R}$. Thus, vertices in level $\ell$ are at distance $\ell$ from $\mathcal{R}$. In particular, $\mathcal{R}$ is in level $0$.

    We construct the cycle inductively.
    Claim \ref{rotarvecinos da un ciclo} assures the existence of a cycle in $K_{s-stab}(n, k)$ for each class $\mathcal{A}$.
    For each class $\mathcal{A}$, let $A_1,A_2,\ldots,A_{|\mathcal{A}|}$
    be the vertices in the class such that 
    \[
A_1A_2\ldots A_{|\mathcal{A}|}A_1
    \]
    is said cycle, and such that $A_i$ is adjacent to $B_i$ in $K_{s-stab}(n, k)$ if $\mathcal{A}$ is adjacent to $\mathcal{B}$ in $T$ (with $i$ computed modulo $|\mathcal{A}|$ and $|\mathcal{B}|$, respectively).

Let $\mathcal{A}^1,\ldots,\mathcal{A}^{\deg_T(\mathcal{R})}$ be the vertices adjacent to $\mathcal{R}$ in level $1$.
As $\deg_T(\mathcal{R})\leq |\mathcal{R}|$, we know that the edges $R_iR_{i+1}$ and $R_j,R_{j+1}$ are different if $1\leq i<j\leq |\mathcal{R}|$.
For $1\leq i \leq \deg_T(\mathcal{R})$, change edges $R_iR_{i+1}$ and $A^i_iA^i_{i+1}$ for the edges $R_iA^i_i$ and $R_{i+1}A^i_{i+1}$.
This generates a cycle containing all vertices in classes from levels $0$ and $1$.

In order to connect vertices from level $\ell$ to level $\ell+1$, let  $\mathcal{B}$
be a vertex level $\ell$,  and let $\mathcal{C}^{1},\ldots,\mathcal{C}^{|\mathcal{B}|-1}$ be the vertices adjacent to $\mathcal{B}$ at level $\ell+1$. Let $B_iB_{i+1}$ be the edge that was deleted to connect vertices in the class $\mathcal{B}$ to their neighbor at level $\ell-1$.
As $\deg_T(\mathcal{B}^j)\leq |\mathcal{B}^j|$, if
$1\leq \alpha<\beta\leq |\mathcal{B}^j|-1$
then
$B_{i+\alpha}B_{i+\alpha+1}$ and $B_{i+\beta}B_{i+\beta+1}$
are two distinct edges in our current graph, where addition is done
modulo $|\mathcal{B}|$.
Then, for each 
$1\leq \alpha\leq |\mathcal{B}|-1$, exchange edges
$B_{i+\alpha}B_{i+\alpha+1}$ and $C^{\alpha}_{i+\alpha}C^{\alpha}_{i+\alpha+1}$
for edges 
$B_{i+\alpha}C^{\alpha}_{i+\alpha}$ and 
$B_{i+\alpha+1}C^{\alpha}_{i+\alpha+1}$, with addition done modulo $|\mathcal{B}|$ or modulo $|\mathcal{C}^{\alpha}|$, according of the vertex.
Therefore, after repeating this process for all levels of the rooted tree, we obtain a Hamiltonian cycle for $K_{s-stab}(n, k)$.
\end{proof}

In order to find the spanning tree needed to apply Claim \ref{C:si spanning de grado bajo, entonces hamiltoniano}, we are going to focus only in a particular set of edges. Given a class $\mathcal{A}= ( a_1,a_2, \dots, a_{i-1},a_i,a_{i+1} \dots ,a_k )$, every class of the form $\mathcal{B}=( a_1,a_2, \dots, a_{i-1}+1,a_i-1,a_{i+1} \dots ,a_k )$ is called a friend class of $\mathcal{A}$. I.e., $\mathcal{B}$ is a friend class of $\mathcal{A}$ if it can be obtained from it by adding $1$ to $a_i$ and  subtracting $1$ to the $a_{i+1}$ for some $1\leq i\leq k$. Notice
that for $\mathcal{B}$ a class, $a_{i+1}$ must be at least $s$. Thus, each class has at most $k$ friends, but some classes have fewer.

Notice that if we start with a vertex from the class $( a_1,a_2, \dots, a_{i-1},a_i,a_{i+1} \dots ,a_k )\in CK_{s-stab}(n, k)$, as shown in (I) below. Then we add 1 to ``$a_{i-1}$'' and subtract one from ``$a_i$'', obtained in (II) by moving the ``$x$'' between those positions, one place to the right.

\begin{figure}[h]
\centering
\begin{tikzpicture}[x=.5cm,y=.5cm]
\draw[step=.5cm,very thin] (0,0) grid (4,1);
\draw[step=.5cm,very thin] (4,0) grid (6,0);
\draw[step=.5cm,very thin] (4,1) grid (6,1);
\draw[step=.5cm,very thin] (6,0) grid (13,1);
\draw[step=.5cm,very thin] (13,0) grid (15,0);
\draw[step=.5cm,very thin] (13,1) grid (15,1);
\draw[step=.5cm,very thin] (15,0) grid (15,1);

\node[draw=none,shape=rectangle] at (-1.5,.5) {(I)}; 
\node[draw=none,shape=rectangle] at (0.5,.5) {x};   
\node[draw=none,shape=rectangle] at (3.5,.5) {x};      

\node[draw=none,shape=rectangle] at (5,.5) {\dots};

\node[draw=none,shape=rectangle] at (6.5,.5) {x}; 
\node[draw=none,shape=rectangle] at (9.5,.5) {x};   
\node[draw=none,shape=rectangle] at (12.5,.5) {x};  
\node[draw=none,shape=rectangle] at (14,.5) {\dots};

\node[draw=none,shape=rectangle] at (21.2,.5) {$\in (a_1,a_2,\dots,a_{i-1},a_i,a_{i+1},\dots,a_k)$};

\node[draw=none,shape=rectangle] at (2,-.5) {$\underbrace{ \ \quad \quad }_{a_1}$};  
\node[draw=none,shape=rectangle] at (8,-.5) {$\underbrace{ \ \quad \quad }_{a_{i-1}}$};  
\node[draw=none,shape=rectangle] at (11,-.5) {$\underbrace{\ \quad \quad  }_{a_i}$};  


\draw[step=.5cm,very thin] (0,-3) grid (4,-2);
\draw[step=.5cm,very thin] (4,-3) grid (6,-3);
\draw[step=.5cm,very thin] (4,-2) grid (6,-2);
\draw[step=.5cm,very thin] (6,-3) grid (13,-2);
\draw[step=.5cm,very thin] (13,-3) grid (15,-3);
\draw[step=.5cm,very thin] (13,-2) grid (15,-2);
\draw[step=.5cm,very thin] (15,-3) grid (15,-2);

\node[draw=none,shape=rectangle] at (-1.5,-2.5) {(II)}; 
\node[draw=none,shape=rectangle] at (0.5,-2.5) {x};   
\node[draw=none,shape=rectangle] at (3.5,-2.5) {x};      

\node[draw=none,shape=rectangle] at (5,-2.5) {\dots};

\node[draw=none,shape=rectangle] at (6.5,-2.5) {x}; 
\node[draw=none,shape=rectangle] at (10.5,-2.5) {x};   
\node[draw=none,shape=rectangle] at (12.5,-2.5) {x};  
\node[draw=none,shape=rectangle] at (14,-2.5) {\dots};

\node[draw=none,shape=rectangle] at (22.5,-2.5) {$\in (a_1,a_2,\dots,a_{i-1}+1,a_i-1,a_{i+1},\dots,a_k)$};

\node[draw=none,shape=rectangle] at (2,-3.5) {$\underbrace{ \ \quad \quad }_{a_1}$};  
\node[draw=none,shape=rectangle] at (8.5,-3.5) {$\underbrace{ \qquad \qquad }_{a_{i-1}+1}$};  
\node[draw=none,shape=rectangle] at (11.5,-3.5) {$\underbrace{}_{a_i-1}$};  

\draw[step=.5cm,very thin] (0,-6) grid (5,-5);
\draw[step=.5cm,very thin] (5,-6) grid (7,-6);
\draw[step=.5cm,very thin] (5,-5) grid (7,-5);
\draw[step=.5cm,very thin] (7,-6) grid (15,-5);

\node[draw=none,shape=rectangle] at (-1.5,-5.5) {(III)}; 
\node[draw=none,shape=rectangle] at (1.5,-5.5) {x};   
\node[draw=none,shape=rectangle] at (4.5,-5.5) {x};      

\node[draw=none,shape=rectangle] at (6,-5.5) {\dots};

\node[draw=none,shape=rectangle] at (7.5,-5.5) {x}; 
\node[draw=none,shape=rectangle] at (11.5,-5.5) {x};   
\node[draw=none,shape=rectangle] at (13.5,-5.5) {x};  
\node[draw=none,shape=rectangle] at (14.5,-5.5) {\dots};

\node[draw=none,shape=rectangle] at (22.5,-5.5) {$\in (a_1,a_2,\dots,a_{i-1}+1,a_i-1,a_{i+1},\dots,a_k)$};

\node[draw=none,shape=rectangle] at (3,-6.5) {$\underbrace{ \ \quad \quad }_{a_1}$};  
\node[draw=none,shape=rectangle] at (9.5,-6.5) {$\underbrace{ \qquad \qquad }_{a_{i-1}+1}$};  
\node[draw=none,shape=rectangle] at (12.5,-6.5) {$\underbrace{}_{a_i-1}$};

\end{tikzpicture}
\label{v136}
\end{figure}

After that, we move every ``$x$'' in (II) one position to the right, getting (III), still in the same class. We conclude that (I) is not a neighbor of (II), but (I) is a neighbor of (III). Then, $( a_1,a_2, \dots, a_{i-1},a_i,a_{i+1} \dots ,a_k )$ is adjacent to $( a_1,a_2, \dots, a_{i-1}+1,a_i-1,a_{i+1} \dots ,a_k )$ in $CK_{s-stab}(n, k)$. Thus, we have the following.
\begin{claim}
Friend classes are adjacent in $CK_{s-stab}(n, k)$.
\label{amigvec}
\end{claim}

\begin{definition}
We define $SCK_{s-stab}(n, k)$ as the spanning subgraph of $CK_{s-stab}(n, k)$ induced by the edges between friend classes.
\end{definition}
In order to apply Claim \ref{C:si spanning de grado bajo, entonces hamiltoniano}, we are going to show that $SCK_{s-stab}(n, k)$ is connected, and that the degree of each vertex in $SCK_{s-stab}(n, k)$ is at most the order of the class. We begin by showing the former.
\begin{lemma}\label{SCKconn}
$SCK_{s-stab}(n, k)$ is connected, for $n=sk+r$, $r \geq 1$, $s \geq 2$.
\end{lemma}

\begin{proof}
For any vertex in $K_{s-stab}(n, k)$ the total amount of spaces is $n-k= sk+r-k=(s-1)k+r$.
Consider the vertex $\displaystyle( \underbrace{s-1,s-1, \dots, s-1}_{k-1}, s-1+r)\in SCK_{s-stab}(n, k)$.

We would like to prove that this vertex is connected to any vertex through edges of the graph.
Let $(a_1,a_2,a_3, \cdots, a_{k-1},a_k)$ be a vertex such that $\displaystyle \sum_{i=1}^{k} a_i=n-k$.

As shown in Claim \ref{amigvec}, $(s-1,s-1, \dots, s-1, s-1+r)$ is adjacent to $(s-1,s-1, \dots,s-1, s, s-2+r)$. Thus, applying the claim successively, we can find a path from   $(s-1,s-1, \dots, s-1, s-1+r)$ to 
$(s-1,s-1, \dots, s-1,2s-2+r-a_k, a_k)$.
But continuing with this process with the next coordinate, we can find a path to $(s-1,s-1, \dots,s-1, 3s-3+r-a_{k-1}- a_k,a_{k-1}, a_k)$. We can keep on going this way and connect $(s-1,s-1, \dots, s-1, s-1+r)$ to $(a_1,a_2,a_3, \cdots, a_{k-1},a_k)$.
\end{proof}
We are ready to show that the degrees of vertices in $SCK_{s-stab}(n, k)$ are bounded.
\begin{claim}\label{C: el grado esta bien}
    If  $\mathcal{A}\in V(SCK_{s-stab}(n, k))$, then $\deg_{SCK_{s-stab}(n, k)}(\mathcal{A})\leq |\mathcal{A}|$, for $n=sk+r$, $r \geq 1$, $s \geq 2$.
\end{claim}

\begin{proof}

By Claim \ref{cardclase}  $|\mathcal{A}|=\dfrac{n}{d}$, where $d$ is the smallest number such that $\mathcal{A} = (a_1,a_2, \dots , a_{\frac{k}{d}})^d$.
Notice then that when $j-i=p\frac{k}{d}$, with $1\leq p\leq d$, then friend class of $\mathcal{A}$ obtained by adding $1$ to $a_i$ and subtracting $1$ to $a_{i+1}$ is the same as any obtained by adding  $1$ to $a_j$ and subtracting $1$ to $a_{j+1}$. Therefore, $\mathcal{A}$ has at most $\frac{k}{d}$ friend classes.   
Therefore,  by Claim \ref{cardclase}, $\deg_{SCK_{s-stab}(n, k)}(\mathcal{A})=\frac{k}{d}< \frac{n}{d}=|\mathcal{A}|$.
   
    
\end{proof}
We are ready to prove our main result.
\begin{proof}[Proof of Theorem \ref{theo: main}]
When $n=sk$, $K_{s-stab}(n, k)$ is isomorphic to $K_s$. Thus, if $s\geq 3$, it is Hamiltonian.

    Assume now that $n\geq sk+1$, and let $T$ be a spanning tree of $SCK_{s-stab}(n, k)$. 
    By Claim $\ref{C: el grado esta bien}$, $\deg_T(\mathcal{A})\leq |\mathcal{A}|$.
    As $V(SCK_{s-stab}(n, k))=V(CK_{s-stab}(n, k))$, $T$ is a spanning tree of $CK_{s-stab}(n, k)$.  Thus, by Claim \ref{C:si spanning de grado bajo, entonces hamiltoniano}, the $s$-Stable Kneser Graph $K_{s-stab}(n, k)$ is Hamiltonian.
\end{proof}

\section{Acknowledgements}
This work was partially supported by Universidad Nacional de San Luis, grants PROICO 03-0723 and PROIPRO 03-2923, MATH AmSud, grants 21-MATH-05 and 22-MATH-02, and Agencia I+D+I grants PICT 2020-00549 and PICT 2020-04064. The first author is supported by a Ph.D. scholarship funded by Consejo Nacional de Investigaciones
Científicas y Técnicas.

\bibliographystyle{abbrv}
\bibliography{bibfile}

\begin{thebibliography}{10}

\bibitem{aplicaciontopo}
M.~Alishahi and H.~Hajiabolhassan.
\newblock A generalization of gale's lemma.
\newblock {\em Journal of Graph Theory}, 88(2):337--346, 2018.

\bibitem{defestables}
N.~Alon, L.~Drewnowski, and T.~Luczak.
\newblock Stable kneser hypergraphs and ideals in n with the nikodym property.
\newblock In {\em Proc. Amer. Math. Soc}, volume 137, pages 467--471, 2009.

\bibitem{Neighcomplexes}
A.~Bj{\"o}rner and M.~De~Longueville.
\newblock Neighborhood complexes of stable kneser graphs.
\newblock {\em Combinatorica}, 23(1):23--34, 2003.

\bibitem{Braun}
B.~Braun.
\newblock Symmetries of the stable kneser graphs.
\newblock {\em Advances in Applied Mathematics}, 45(1):12--14, 2010.

\bibitem{Indepcomplex}
B.~Braun.
\newblock Independence complexes of stable kneser graphs.
\newblock {\em The Electronic Journal of Combinatorics}, 18(1):P118, 2011.

\bibitem{Chen}
Y.-C. Chen.
\newblock Kneser graphs are hamiltonian for $n\geq 3k$.
\newblock {\em Journal of Combinatorial Theory, Series B}, 80(1):69--79, 2000.

\bibitem{paper}
A.~V. Ledezma, A.~Pastine, P.~Torres, and M.~Valencia-Pabon.
\newblock On the diameter of schrijver graphs.
\newblock {\em arXiv preprint arXiv:2112.01884}, 2021.

\bibitem{Lov}
L.~Lov{\'a}sz.
\newblock Kneser's conjecture, chromatic number, and homotopy.
\newblock {\em Journal of Combinatorial Theory, Series A}, 25(3):319--324,
  1978.

\bibitem{MMN}
A.~Merino, T.~M{\"u}tze, and Namrata.
\newblock Kneser graphs are hamiltonian.
\newblock In {\em Proceedings of the 55th Annual ACM Symposium on Theory of
  Computing}, pages 963--970, 2023.

\bibitem{Meu}
F.~Meunier.
\newblock The chromatic number of almost stable kneser hypergraphs.
\newblock {\em Journal of Combinatorial Theory, Series A}, 118(6):1820--1828,
  2011.

\bibitem{Mut}
T.~M{\"u}tze.
\newblock Proof of the middle levels conjecture.
\newblock {\em Proceedings of the London Mathematical Society},
  112(4):677--713, 2016.

\bibitem{mutze2024hamiltonicity}
T.~M{\"u}tze and Namrata.
\newblock Hamiltonicity of schrijver graphs and stable kneser graphs.
\newblock {\em arXiv preprint arXiv:2401.01681}, 2024.

\bibitem{MNW}
T.~M{\"u}tze, J.~Nummenpalo, and B.~Walczak.
\newblock Sparse kneser graphs are hamiltonian.
\newblock {\em Journal of the London Mathematical Society}, 103(4):1253--1275,
  2021.

\bibitem{Sch}
A.~Schrijver.
\newblock Vertex-critical subgraphs of kneser-graphs.
\newblock {\em Nieuw Archief voor Wiskunde}, 26(3):454--461, 1978.

\bibitem{SS}
I.~Shields and C.~D. Savage.
\newblock A note on hamilton cycles in kneser graphs.
\newblock {\em Bull. Inst. Combin. Appl}, 40:13--22, 2004.

\bibitem{Tor}
P.~Torres.
\newblock The automorphism group of the s-stable kneser graphs.
\newblock {\em Advances in Applied Mathematics}, 89:67--75, 2017.

\bibitem{TorPav}
P.~Torres and M.~Valencia-Pabon.
\newblock Shifts of the stable kneser graphs and hom-idempotence.
\newblock {\em European Journal of Combinatorics}, 62:50--57, 2017.

\end{thebibliography}
\end{document}